\documentclass[11pt,reqno]{amsart}

\usepackage{amsmath,amsfonts,amsthm,amssymb,amsxtra}
\usepackage{amsxtra, amssymb, mathrsfs}
\usepackage[usenames,dvipsnames]{color}

\numberwithin{equation}{section}

\newtheorem{thm}{Theorem}[section]
\newtheorem{cor}[thm]{Corollary}
\newtheorem{lem}[thm]{Lemma}
\newtheorem{prop}[thm]{Proposition}

\theoremstyle{remark}
\newtheorem{rem}[thm]{\bf Remark}

\newcommand{\eps}{\varepsilon}
\renewcommand{\epsilon}{\varepsilon}
\newcommand{\N}{\mathbb{N}}

\newcommand{\R}{\mathbb{R}}
\newcommand{\s}{\mathcal{S}}

\title[Weak perturbations of the p--Laplacian]{Weak perturbations of the p--Laplacian}

\author{Tomas Ekholm}
\address{Tomas Ekholm, Department of Mathematics,
Royal Institute of Technology, S-100 44 Stockholm, Sweden}
\email{tomase@math.kth.se}

\author{Rupert L. Frank}
\address{Rupert L. Frank, Mathematics 253-37, Caltech, Pasadena, CA 91125, USA}
\email{rlfrank@caltech.edu}

\author{Hynek Kova\v r\'{\i}k}
\address{Hynek Kova\v r\'{\i}k, DICATAM, Sezione di Matematica, Universit\`a degli studi di Brescia,
Via Branze, 38 - 25123  Brescia, Italy}
\email{hynek.kovarik@ing.unibs.it}

\begin{document}

\begin{abstract}
We consider the p-Laplacian in $\R^d$ perturbed by a weakly coupled potential. We calculate the asymptotic expansions of the lowest eigenvalue of such an operator in the weak coupling limit separately for $p>d $ and $p=d$ and discuss the connection with Sobolev interpolation inequalities. 
\end{abstract}

\keywords{p-Laplacian, weak coupling, Sobolev inequality}

\thanks{\copyright\, 2013 by the authors. This paper may be
reproduced, in its entirety, for non-commercial purposes.}

\maketitle

 {\bf AMS Mathematics Subject Classification:}  49R05, 35P30 \\

{\bf  Keywords:} p-Laplacian, weak coupling, Sobolev inequalities \\


\section{Introduction}

In this paper we consider the functional 
\begin{equation} \label{p-laplace}
Q_V [u] = \int_{\R^d}\, \left(|\nabla u|^p- V\, |u|^p\right)\, dx, 
\qquad u\in W^{1,p}(\R^d), \quad p >1,
\end{equation}
with a given function $V:\R^d\to \R$ which is assumed to vanish at infinity in a sense to be made precise. We are interested in the minimization problem 
\begin{equation} \label{var-problem}
\lambda(V) =  \inf_{u\in W^{1,p}(\R^d)} \frac{Q_V [u]}{ \int_{\R^d}\,  |u|^p\, dx } \,.
\end{equation}
If \eqref{var-problem} admits a minimizer $u$, then the latter satisfies in the weak sense the non-linear eigenvalue equation
\begin{equation} \label{eq:groundstate}
-\Delta_p(u) - V |u|^{p-2}\, u = \lambda(V)\, |u|^{p-2}\, u \,, 
\end{equation}
where $-\Delta_p(u) := -\nabla \cdot(|\nabla u|^{p-2}\, \nabla u)$ is the $p$-Laplacian. Equation \eqref{eq:groundstate} is a particular case of a quasilinear differential problem and we refer to the monographs \cite{lu,ps2} and to \cite{s,s2, tr} for the general theory of such equations. The $p$-Laplacian equation with a zero-th order term $V$ has attracted particular attention. Existence of positive solutions to the equation $-\Delta_p(u) = V |u|^{p-2}\, u$  and related regularity questions were studied in \cite{ps, pt2, tt, to, pt}. For the discussion of maximum and comparison principles and positive Liouville theorems, see \cite{gs, ptt}.

\smallskip

In the present paper we are going to study the behaviour of $\lambda(\alpha V)$ for small values of $\alpha$. It is not difficult to see that $\lambda(\alpha V)\to 0$ as $\alpha\to 0$ for all sufficiently regular and decaying $V$. Our goal here is to find the correct asymptotic order and the correct asymptotic coefficient. It turns out that the asymptotic order depends essentially on the relation between the values of the exponent $p$ and the dimension $d$. If $p < d$, then by the Hardy inequality \cite{ok} we have
$$
\int_{\R^d} |\nabla u|^p\, dx \, \geq \Big(\frac{d-p}{p}\Big)^p \int_{\R^d} \frac{|u|^p}{|x|^p}\, dx, \qquad u\in W^{1,p}(\R^d), \quad d>p \,.
$$
Therefore, if $|V(x)| \leq C\, |x|^{-p}$ for some $C>0$, then $\lambda(\alpha V)=0$ for all $\alpha$ small enough. 
However, if $p \geq d$ and $\int_{\R^d}  V >0$, then we have $\lambda(\alpha V)<0$ for any $\alpha>0$. The latter is easily verified by a suitable choice of test functions. Moreover, if $V$ is bounded and compactly supported, then  $\lambda(\alpha V)<0$ for any $\alpha>0$  even when $\int_{\R^d} V =0$, see \cite[Prop. 4.5]{pt}. Consequently, we will always assume that $p\geq p$.

\smallskip

The question about the asymptotic behavior of $\lambda(\alpha V)$ for small $\alpha$ was intensively studied in the linear case $p=2$ (see, e.g., \cite{bgs,kl77,ks,si}), where equation \eqref{eq:groundstate} defines the ground state energy of the Schr\"odinger operator $-\Delta - V$. In particular, it turns out that for sufficiently fast decaying $V$ we have
\begin{equation} \label{eq:p=2}
\sqrt{-\lambda(\alpha V)} \, = \, \frac 12\ \alpha\! \int_{\R} V\, dx - c\, \alpha^2 + o(\alpha^2), \quad \alpha\to 0, \qquad d=1, \ p=2, 
\end{equation}
with an explicit constant $c$ depending on $V$, see \cite{si}. The proof of  \eqref{eq:p=2} is based on the Birman-Schwinger principle and on the explicit knowledge of the unperturbed Green function. With suitable modifications, this method was applied also to Schr\"odinger operators with long-range potentials, \cite{bgs,kl79}, and even to higher order and fractional Schr\"odinger operators \cite{az1,az2,ha}. 

\smallskip

Much less is known about the non-linear case $p\neq 2$ where the operator-theoretic methods developed  for $p=2$ cannot be used.  We will therefore apply a different, purely variational technique which allows us to analyze the asymptotic behaviour of $\lambda(\alpha V)$ for all $p>1$. A similar variational approach has already been used in a linear problem in \cite{fmv}, but here we take it much further into the quasi-linear realm (where, for instance the symmetry reduction that we crucial in \cite{fmv} is no longer available).

We will present our main results separately for $p>d$, see Theorem \ref{thm-main-non-crit}, and for $p=d$, see Theorem \ref{thm-main-crit}. In the case $p>d$ we shall show, in particular, that there is a close relation between the asymptotic behaviour of $\lambda(\alpha V)$ and the Sobolev interpolation inequality (see, e.g., \cite[Thm 5.9]{a})
\begin{equation} \label{sob}
\| u \|_\infty^p \leq\, \s_{d,p} \, \| \nabla u \|_p^{d}\, \,
\| u \|_p^{p- d}\,, \qquad u \in W^{1,p}(\R^d)\,, \quad d< p\,.
\end{equation}
By convention $\s_{d,p}$ will always denote the optimal (that is, smallest possible) constant in \eqref{sob}. On one hand, the constant $\s_{d,p}$ enters into the asymptotic coefficient in the expansion of $\lambda(\alpha V)$, see equation \eqref{limit-1}.  On the other hand, minimizers of problem \eqref{var-problem}, when suitably rescaled and normalised, converge (up to a subsequence) locally uniformly to a minimizer of the Sobolev inequality \eqref{sob} as $\alpha\to 0$, see Proposition \ref{prop:minimizers}. 

The case $p=d$ is much more delicate and requires (slightly) more regularity of the potential $V$ since functions in $W^{1,d}(\R^d)$, which appear in \eqref{var-problem}, are not necessarily bounded. While the case $p>d$ can be dealt with by energy methods (i.e. on the $W^{1,p}(\R^d)$ level of regularity), heavier PDE technics (Harnack's inequality, H\"older continuity bounds) are necessary to deal with $p=d$. The subtly of the case $p=d$ can also be seen in the asymptotic order: while $\lambda(\alpha V)$ vanishes algebraically as $\alpha\to 0$ for $p>d$, it vanishes exponentially fast for $p=d$, see equation \eqref{eq:log}. 

As we shall see, the asymptotic coefficient will depend on $V$ only through $\int_{\R^d} V\,dx$. We emphasize here that we do \emph{not} impose a sign condition on $V$. Thus, the positive and the negative parts of $V$ contribute both to the asymptotic coefficient and there will be cancellations. This is one of main difficulties that we overcome. In fact, if $V$ is non-negative, then the proof is considerably simpler.

A common feature of both Theorems \ref{thm-main-non-crit} and Theorem \ref{thm-main-crit} is that their proofs rely, among other things, on the fact that  minimizers $u_\alpha$ of \eqref{var-problem}, suitably normalized, converge locally uniformly to a constant. While in the case $d< p$ this follows from Morrey's Sobolev inequality and energy considerations, for $d=p$ we have to employ a regularity argument related to the H\"older continuity of  $u_\alpha$, see Lemma \ref{lem-hoelder}, with explicit dependence on the coefficients of the equation.


\section{Main results}

Our main results describe the asymptotics of the infimum $\lambda(\alpha V)$ of the functional $Q_{\alpha V}[u]$ as $\alpha\to 0$, see \eqref{p-laplace} and \eqref{var-problem}. Our first theorem concerns the subcritical case $p>d$.

\begin{thm}\label{thm-main-non-crit}
Let $p>d\geq 1$. Let $V\in L^1(\R^d)$ be such that $\int_{\R^d} V(x)\, dx > 0$. Then
\begin{equation} \label{limit-1}
\lim_{\alpha\to 0+}\, \alpha^{- \frac{p}{p-d}}\ \lambda(\alpha V) \, = \, 
- \frac{p-d}{p}\ \left(\frac dp\right)^{\frac{d}{p-d}}\ 
\left( \s_{d,p} \int_{\R^d} V(x)\, dx\right)^{\frac{p}{p-d}} \,,
\end{equation}
where $\s_{d,p}$ is the sharp constant in the Sobolev inequality \eqref{sob}.
\end{thm}

We also have a theorem that describes the asymptotics of the minimizers of the functional $Q_{\alpha V}[u]$; see Proposition \ref{prop:minimizers}.

In the endpoint case $d=p$ we have

\begin{thm}\label{thm-main-crit}
Let $p=d >1$. Suppose that $V\in L^q(\R^d)\cap L^1(\R^d)$ for some $q>1$ and that $\int_{\R^d} V(x)\,dx > 0$. Then 
\begin{equation} \label{eq:log}
\lim_{\alpha\to 0+}\, \alpha^{\frac{1}{d-1}}\,  \log \frac1{|\lambda(\alpha V)|} = d\ \omega_d^{\frac{1}{d-1}}\, \left(\int_{\R^d}  V(x)\, dx  \right)^{-\frac{1}{d-1}} \,,
\end{equation}
where $\omega_d$ denotes the surface area of the unit sphere in $\R^d$. 
\end{thm}

\begin{rem}
Let us compare the assumptions on $V$ in Theorems \ref{thm-main-non-crit} and \ref{thm-main-crit}. If $p>d$ and $V_+ \notin L^1(\R^d)$, $V_-\in L^1(\R^d)$, then Theorem \ref{thm-main-non-crit} easily implies that
$$
\lim_{\alpha\to 0+}\, \alpha^{-\frac{p}{p-d}}\ \lambda(\alpha V) = - \infty \,.
$$
Thus, at least under the additional hypothesis $V_-\in L^1(\R^d)$, the condition $V_+\in L^1(\R^d)$ is necessary and sufficient for finite asymptotics of $\alpha^{-\frac{p}{p-d}}\ \lambda(\alpha V)$. This is not true for the asymptotics of $\alpha^{\frac{1}{d-1}}\,  \log |\lambda(\alpha V)|^{-1}$ in the case $p=d$, and this is the reason for the additional assumption $V\in L^q(\R^d)$ for some $q>1$. Indeed, we claim that there are $0\leq V\in L^1(\R^d)$ such that $\lambda(\alpha V)=-\infty$ for any $\alpha>0$. To see this, choose $\sigma\in (1,d)$ and consider $V(x) = |x|^{-d} |\log |x||^{-\sigma}$ for $|x|\leq e^{-1}$ and $V(x)= 0$ for $|x|> e^{-1}$. Then $\sigma>1$ implies $V\in L^1(\R^d)$. Since $\sigma<d$ we can choose a $\rho\in [(\sigma-1)/d,(d-1)/d)$ and define $u(x) = |\ln|x||^\rho \zeta(x)$, where the function $\zeta\in C_0^\infty(\R^d)$ equals one in a neighborhood of the origin. Then $\rho<(d-1)/d$ implies that $u\in W^{1,d}(\R^d)$, whereas $\rho\geq (\sigma-1)/d$ implies that $\int_{\R^d} V|u|^d \,dx = \infty$. Thus, $Q_{\alpha V}[u]=-\infty$ for any $\alpha>0$.  
\end{rem}

\begin{rem}
In the quadratic case $p=2$, Theorems \ref{thm-main-non-crit} and \ref{thm-main-crit} recover the asymptotics originally found in \cite{si} using a different, operator theoretic approach. Both \eqref{limit-1} and \eqref{eq:log} were originally proved in \cite{si} under more restictive conditions on $V$. For $d=1$ these restrictions were later removed in \cite[Sec.4]{kl77}; note also that according to Lemma \ref{sobminimizer} below we have $\s_{1,2}=1$ for $p=2$ and $d=1$.
\end{rem}

While our theorems give a complete answer in the case $V\in L^1(\R^d)$ (plus additional assumptions if $p=d$) with $\int_{\R^d} V\,dx>0$, the following questions, which we consider interesting, remain open:
\begin{enumerate}
\item What happens if $V\in L^1(\R^d)$ (plus some additional assumptions), but $\int_{\R^d} V\,dx=0$? For results in the case $p=2$, see \cite{si,kl77,bcez}.
\item What happens if $V\notin L^1(\R^d)$, but $V(x)= |x|^{-\sigma}(1+o(1))$ as $|x|\to\infty$ with $0<\sigma\leq d$?  For results in the case $p=2$, see \cite{kl79}.
\end{enumerate}

The proofs of Theorems \ref{thm-main-non-crit}  and \ref{thm-main-crit} are given in Sections \ref{sect:p<d} and \ref{sect:p=d} respectively.

\subsection*{Notation}  Given $r>0$ and a point $x\in\R^d$ we denote by $B(r,x)\subset\R^d$ the open ball with radius $r$ centred in $x$. If $x=0$, then we write $B_r$ instead of $B(r,0)$. Furthermore, given a set $\Omega \subset \R^d$ we denote by $\Omega^c$ its complement in $\R^d$.  The $L^q$ norm of a function $u$ in $\Omega$ will be denoted by $\|u\|_{L^q(\Omega)}$  if $\Omega\neq \R^d$ and by $\|u\|_q$ if $\Omega =\R^d$. 

\section{Case $d <  p$} \label{sect:p<d}

Before we proceed with the proof of Theorem \ref{thm-main-non-crit} we give some preliminary results concerning Sobolev inequality \eqref{sob} and the properties of the functional $Q_V [u]$. 

\subsection{Sobolev inequality}

We recall that $\s_{d,p}$ denotes the optimal constant in the Sobolev interpolation inequality \eqref{sob}. In this subsection we discuss a closely related (and, in fact, equivalent, as we shall show) minimization problem which depends on a parameter $v>0$ in addition to an exponent $q>d\geq 1$. We define
\begin{equation}
\label{eq:e}
E(v) = \inf_{\|u\|_p =1} \left( \|\nabla u\|_p^p - v |u(0)|^p \right).
\end{equation}
(Note that by the Sobolev embedding theorem any function in $W^{1,q}(\R^d)$, $q>d$, has a continuous representative and therefore $u(0)$ is unambiguously defined. The following lemma shows, in particular, that $E(v)>-\infty$.

\begin{lem}\label{soblemma}
Let $p>d\geq 1$ and $v>0$. Then
$$
E(v) = - \frac{p-d}{p}\, \left(\frac dp\right)^{\frac{d}{p-d}}\,
\left(\s_{d,p} v\right)^{\frac{p}{p-d}} \,.
$$
Moreover, the infimum is attained by a non-negative, symmetric decreasing function. Finally, any minimizing sequence is relatively compact in $W^{1,p}(\R^d)$.
\end{lem}

We include a proof of this lemma for the sake of completeness.

\begin{proof}
By the Sobolev inequality \eqref{sob} we have
$$
|u(0)|^p \leq \|u\|_\infty^p \leq \mathcal S_{d,p} \|\nabla u \|_p^d \|u\|_p^{p-d}
$$
and, therefore, if $\|u\|_p=1$,
\begin{align*}
\|\nabla u\|_p^p - v |u(0)|^p
\geq \|\nabla u\|_p^p - v \mathcal S_{d,p} \|\nabla u \|_p^d
& \geq \inf_{X\geq 0} \left( X^p - v \mathcal S_{d,p} X^d \right) \\
& = - \frac{p-d}{p}\, \left(\frac dp\right)^{\frac{d}{p-d}}\,
\left(\s_{d,p} v\right)^{\frac{p}{p-d}} \,.
\end{align*}
This shows that $E(v) \geq -\frac{p-d}{p}\, \left(\frac dp\right)^{\frac{d}{p-d}}\, \left(\s_{d,p} v\right)^{\frac{p}{p-d}}$. In particular, $E(v)>-\infty$.

To prove the reverse inequality, we first note that, by scaling,
$$
E(v) = E(1) \,v^{\frac{p}{p-d}} \,.
$$
(To see this, write $u$ in the form $u(x) = v^{\frac{d}{p(p-d)}} w(v^{\frac{1}{p-d}} x)$.) We note also that $E(v)<0$. (Indeed, for a fixed $u\in W^{1,p}(\R^d)$ with $\|u\|_p=1$ and $u(0)\neq 0$ we clearly have $\|\nabla u\|_p^p - v |u(0)|^p\to-\infty$ as $v\to\infty$ and therefore $E(v)<0$ for all sufficiently large $v$. By the scaling law, this implies that $E(v)<0$ for any $v$.)

 Now let $u\in W^{1,p}(\R^d)$. Then, by the Sobolev embedding theorem $u$ can be assumed to be continuous and vanishing at infinity, so there is an $a\in\R^d$ such that $|u(a)|=\|u\|_\infty$. Let $\tilde u(x) = u(x+a)/\|u\|_p$. Then, by the definition of $E(v)$,
$$
\|\nabla \tilde u\|^p_p - v |\tilde u(0)|^p \geq E(v) \,,
$$
i.e.,
$$
\|\nabla u\|_p^p  \geq v \|u\|_\infty^p + E(v) \|u\|_p^p
=  v\|u\|_\infty^p + E(1)\, \,v^{\frac{p}{p-d}} \|u\|_p^p \,.
$$
Since this is true for any $v>0$ we have
\begin{align*}
\|\nabla u\|_p^p  \geq v \|u\|_\infty^p + E(v) \|u\|_p^p
& \geq \sup_{v>0} \left( v\|u\|_\infty^p + E(1)\, \,v^{\frac{p}{p-d}} \|u\|_p^p \right) \\
& = \|u\|_\infty^{\frac{p^2}{d}} \|u\|_p^{-\frac{p(p-d)}{d}} |E(1)|^{-\frac{p-d}{d}} \left( \frac{p-d}{p} \right)^\frac{p-d}{d} \frac{d}{p} \,.
\end{align*}
This proves that $\s_{d,p} \leq |E(1)|^{\frac{p-d}{p}} \left( \frac{p-d}{p} \right)^{-\frac{p-d}{p}} \left(\frac{d}{p}\right)^\frac dp$.

We next prove that any minimizing sequence is relatively compact in $W^{1,p}(\R^d)$. Let $(u_n)\subset W^{1,p}(\R^d)$ be a minimizing sequence for $E(v)$. Using the bounds in the first part of the proof it is easy to see that $(u_n)$ is bounded in $W^{1,p}(\R^d)$ and therefore, after passing to a subsequence if necessary, we may assume that $u_n$ converges weakly in $W^{1,p}(\R^d)$ to some $u\in W^{1,p}(\R^d)$. By weak convergence,
\begin{equation}
\label{eq:comp}
\liminf_{n\to\infty} \|\nabla u_n\|_p^p \geq \|\nabla u\|_p^p \,,
\qquad
1 \geq \liminf_{n\to\infty} \|u_n\|_p^p \geq \|u\|_p^p \,,
\end{equation}
and, by the Rellich--Kondrashov theorem (see, e.g., \cite[Thm. 8.9]{ll}), $u_n(0)\to u(0)$. We conclude that
$$
0> E(v) = \lim_{n\to\infty} \left( \|\nabla u_n\|_p^p - v |u_n(0)|^p \right) \geq \|\nabla u\|_p^p - v|u(0)|^p \geq E(v) \|u\|_p^p \,.
$$
This, together with the second assertion in \eqref{eq:comp} implies that $\|u\|_p = 1$. Together with the first assertion in \eqref{eq:comp} and the convergence of $u_n(0)$ it also implies that $\|\nabla u_n\|_p\to \|\nabla u\|_p$. Thus, $u_n$ converges in fact strongly to $u$ in $W^{1,p}(\R^d)$.

Thus, we have shown that there is a minimizer. In view of the rearrangement inequalities $\|\nabla u^*\|_p \leq \|\nabla u\|_p$, $\|u^*\|_p=\|u\|_p$ and $|u^*(0)|\geq |u(0)|$ (see, e.g.,~\cite{tal} and \cite[Thm. 3.4]{ll}) we see that among the minimizers there is a non-negative, symmetric decreasing function. This concludes the proof.
\end{proof}

\begin{rem}\label{soblemmarem}
It is easy to see that
$$
E(v) = \inf_{\|u\|_p=1} \left( \|\nabla u\|_p^p - v \|u\|_\infty^p \right) .
$$
This will be useful in the following.
\end{rem}

In one dimension we can compute the value of the sharp constant $\s_{d,p}$ in \eqref{sob}.

\begin{lem} \label{sobminimizer}
If $d=1$, then $\s_{1,p} = \frac p2$ for any $p>1$.
\end{lem}

\begin{proof}
Let $u$ be the (symmetric decreasing) optimizer for $E(v)$. The Euler--Lagrange equation reads
\begin{equation} \label{euler-lagrange}
(p-1) \, u''(x)\,  (-u'(x))^{p-2} = \lambda u(x)^{p-1} \qquad \text{in}\ (0,\infty) \,,
\end{equation}
together with the boundary condition
$$
2 (-u'(0+))^{p-1} = v u(0)^{p-1} \,.
$$
Multiplying \eqref{euler-lagrange} by $u'$ we obtain
$$
\left( (p-1) (-u')^p - \lambda u^p \right)' = 0
\qquad \text{in}\ (0,\infty) \,.
$$
Since $u\in W^{1,p}(\R^d)$ we have $u(x)\to 0$ as $x\to\infty$. Since $(p-1) (-u')^p - \lambda u^p$ is constant, $\lim_{x\to\infty} u'(x)$ exists as well and, therefore, needs to be zero. Thus
\begin{equation}\label{euler-lagrangemult}
(p-1) (-u')^p - \lambda u^p = 0
\qquad \text{in}\ (0,\infty) \,.
\end{equation}
Note that this shows that $\lambda>0$. Moreover, we obtain
$$
-u' = \left( \frac{\lambda}{p-1} \right)^{\frac1p} u
\qquad\text{in}\ (0,\infty) \,,
$$
and, thus,
$$
u(x) = u(0) \exp\left( - \left( \frac{\lambda}{p-1} \right)^{\frac1p} x \right) 
\qquad\text{in}\ (0,\infty) \,.
$$
The boundary condition implies that $\lambda = (p-1)(v/2)^{p/(p-1)}$. We conclude that
$$
E(v) = \frac{2\int_0^\infty |u'|^p\,dx - v u(0)^p}{2 \int_0^\infty u^p \,dx} 
= -(p-1) \left( \frac{v}{2} \right)^{\frac{p}{p-1}} \,.
$$
By Lemma \ref{soblemma} this implies the assertion.
\end{proof}

\subsection{Preliminaries}

\begin{lem} \label{weak-cont}
Let $p>d$ and assume that $V\in L^1(\R^d)$. Then for any $u\in W^{1,p}(\R^d)$,
\begin{equation} \label{low-sharp}
Q_{V}[ u] \, \geq \, - 
\frac{p-d}{p}\ \left(\frac dp\right)^{\frac{d}{p-d}}\ 
\left( \s_{d,p} \int_{\R^d} V_+\, dx\right)^{\frac{p}{p-d}}\  \|u\|_p^p.
\end{equation}
Moreover, $Q_V [ u]$ is
weakly lower semi-continuous in $W^{1,p}(\R^d)$. 
\end{lem}

\begin{proof}
For any $u\in W^{1,p}(\R^d)$,
$$
Q_V[u] \geq \|\nabla u\|_p^p - \int_{\R^d} V_+ \,dx\ \|u\|_\infty^p
\geq E\left( \int_{\R^d} V_+ \,dx \right) \,.
$$
The second inequality used Remark \ref{soblemmarem}. The first assertion now follows from Lemma \ref{soblemma}.

To prove weak lower semi-continuity assume that $(u_j)$ converges weakly in $W^{1,p}(\R^d)$ to some $u$. Then
the sequence $(u_j)$ is bounded in $W^{1,p}(\R^d)$ and hence, by \eqref{sob}, in $L^\infty(\R^d)$.  We have
\begin{equation} \label{eq:aux1}
\Big | \int_{\R^d} V (|u_j|^p-|u|^p)\, dx \Big | \leq
\|u_j - u\|_{L^\infty(B_R)} \|f_j\|_\infty \|V\|_1
+ 2\left( \sup_j \|u_j\|_\infty^p \right) \,  \|V\|_{L^1({B_R^c})},
\end{equation}
where 
$f_j := (|u_j|^p - |u|^p)/(|u_j|-|u|)$ satisfies $|f_j|\leq p\max\{|u_j|^{p-1},|u|^{p-1}\}$ and is therefore bounded. Since the sequence $(u_j)$ is bounded in $W^{1,p}(\R^d)$, inequality
\eqref{sob} implies that $\|f_j\|_\infty$ is bounded uniformly with
respect to $j$. On the other hand, the
Rellich-Kondrashov theorem (see, e.g., \cite[Thm.8.9]{ll}) says that $(u_j)$
converges to $u$ uniformly on compact subsets of $\R^d$.
Hence, sending first $j\to \infty$ and then $R\to
\infty$ in \eqref{eq:aux1} shows that the functional $\int_{\R^d}\,
V\, |u|^p\, dx$ is weakly continuous on $W^{1,p}(\R^d)$. Since
$\|\nabla u\|_p^p$ is weakly lower semi-continuous, due to the fact
that $p > 1$, the same is true for $Q_V[u]$.
\end{proof}

\begin{rem}
Note that inequality \eqref{low-sharp} yields the lower bound in \eqref{limit-1} in the case $V\geq 0$.
\end{rem}

\begin{cor} \label{cor-minimizer}
Let $V\in L^1(\R^d)$ and $p>d$. Assume that $\lambda(V)<0$. Then there is a non-negative function $u\in W^{1,p}(\R^d)$ such that
\begin{equation} \label{minimum}
\lambda(V)= \frac{Q_V[ u]}{\|u\|_p^p}.
\end{equation}
\end{cor}

\begin{proof}
Let $(u_j)$ be a minimizing sequence for $Q_{V}$, normalized such that $\|u_j\|_p=1$ for any $j\in\N$.
Since $\lambda( V)<0$, we may assume without loss of
generality that $Q_V[ u_j]<0$ for any
$j\in\N$. Hence with the help of \eqref{sob} we get
\begin{equation} \label{aux}
\|\nabla u_j\|_p^p < \, \int_{\R^d}\,
V_+ \, |u_j|^p\, dx \, \leq \, \|V_+\|_1\, \|u_j\|^p_\infty
\leq \,\s_{d,p}\, \|V_+\|_1\, \|\nabla u_j\|_p^d \,.
\end{equation}
Since $p>d$,
it follows that the sequence $(u_j)$ is bounded in
$W^{1,p}(\R^d)$ and, after passing to a subsequence if necessary, 
we may assume that $(u_j)$ converges
weakly in $W^{1,p}(\R^d)$ to some $u \in W^{1,p}(\R^d)$. The weak convergence implies
$$
\|u\|_p\leq \liminf_{j\to\infty} \|u_j\|_p = 1.
$$
Since $Q_{V}[ u]$ is weakly lower semicontinuous by Lemma
\ref{weak-cont}, the above inequality implies 
$$
0>\lambda( V) = \lim_{j\to\infty}\, Q_V[ u_j] \geq
Q_V[ u] \geq \lambda( V)\, \|u\|_p^p \geq
\lambda( V).
$$
This implies that $Q_V[u] = \lambda( V) $ and $\|u\|_p=1$, i.e., $u$ is a minimizer for the problem \eqref{var-problem}.

Since $u\in W^{1,p}(\R^d)$ implies $|u|\in W^{1,p}(\R^d)$ with $|\nabla |u|| = |\nabla u|$ almost everywhere (see, e.g., \cite[Thm. 6.17]{ll}), we may choose $u$ non-negative.  
\end{proof}

\subsection{Proof of Theorem \ref{thm-main-non-crit}. Upper bound}
\label{sect-upperb-1}


For any fixed function $\varphi\in W^{1,p}(\R^d)$ with $\|\varphi\|_p=1$ we define
$$
v_\alpha(x) := \alpha^{\frac{d}{p(p-d)}} \, \varphi(\alpha^{\frac{1}{p-d}} x), \quad \alpha>0 \,.
$$
Then $\|v_\alpha\|_p=1$ for all $\alpha>0$ and
\begin{align*}
\lambda(\alpha V) \leq Q_{\alpha V}[v_\alpha] 
= \alpha^{\frac{p}{p-d}} \Big( \|\nabla \varphi\|_p^p - \int_{\R^d}\, V(x) |\varphi(\alpha^{\frac{1}{p-d}} x)|^p\, dx\Big) \,.
\end{align*}
Since $\varphi\in W^{1,p}(\R^d)$, the Sobolev embedding implies that $\varphi\in C(\R^d)\cap L^\infty(\R^d)$ and therefore, by dominated convergence,
$$
\int_{\R^d}\, V(x) |\varphi(\alpha^{\frac{1}{p-d}} x)|^p\, dx
\to \int_{\R^d}\, V\, dx \ |\varphi(0)|^p 
\qquad\text{as}\ \alpha\to 0\,.
$$
Since $\varphi$ is arbitrary, we have shown that
$$
\limsup_{\alpha\to 0+} \alpha^{\frac{p}{d-p}} \lambda(\alpha V) = \inf_{\|\varphi\|_p=1} \left( \|\nabla \varphi\|_p^p - \int_{\R^d}\, V\,dx\ |\varphi(0)|^p \right) = E\left(\int_{\R^d}\, V\,dx \right) \,.
$$
The upper bound in Theorem \ref{thm-main-non-crit} now follows from Lemma \ref{soblemma}.


\subsection{Proof of Theorem \ref{thm-main-non-crit}. Lower bound}

It follows from the proof of the upper bound that $\lambda(\alpha V)<0$ for all sufficiently small $\alpha>0$ and hence, by Corollary \ref{cor-minimizer}, for all such $\alpha$ there is a non-negative minimizer $u_\alpha$ of the problem \eqref{var-problem}. (It is easy to see that, in fact, $\lambda(\alpha V)<0$ for \emph{all} $\alpha>0$. Indeed, $\alpha^{-1} Q_{\alpha V}[u]$ is non-increasing for every $u\in W^{1,p}(\R^d)$ and therefore $\alpha^{-1} \lambda(\alpha V)$ is non-increasing. Thus, if it is negative for some $\alpha>0$, it is negative for all larger $\alpha$'s.)

We normalize $u_\alpha$ so that $\|u_\alpha\|_p=1$. The key step in the proof is to show that
\begin{equation}
\label{eq:minimizersgoal}
\lim_{\alpha\to 0+} \alpha^{-\frac d{p-d}} \int_{\R^d} V(x)\, \left ( u_\alpha(x)^p- u_\alpha(0)^p\right )\, dx = 0 \,.
\end{equation}
Assuming this for the moment, let us complete the proof. We define
\begin{equation} \label{scaling}
f_\alpha(x) = \alpha^{-\frac{d}{p(p-d)}}\, u_\alpha\Big(x\,
\alpha^{-\frac{1}{p-d}}\Big )
\end{equation}
and observe that $\|f_\alpha\|_p=1$ and
$$
\|\nabla f_\alpha \|_p^p- \int_{\R^d} V_\alpha(x)\,  f_\alpha(x)^p\, dx
= \alpha^{-\frac{p}{p-d}}\  Q_{\alpha V}[u_\alpha] \,,
$$
where $V_\alpha(x) = \alpha^{-d/(p-d)} V(x\, \alpha^{-1/(p-d)})$. Since \eqref{eq:minimizersgoal} can be rewritten as
$$
\lim_{\alpha\to 0} \left( \int_{\R^d} V_\alpha(x)\, f_\alpha(x)^p\,dx - \int_{\R^d} V\,dx\  f_\alpha(0)^p \right ) = 0 \,,
$$
we obtain
\begin{align}\label{eq:lowerboundmin}
\liminf_{\alpha\to 0+} \alpha^{-\frac{p}{p-d}}\ \lambda(\alpha V)
& = \liminf_{\alpha\to 0+} \alpha^{-\frac{p}{p-d}}\ Q_{\alpha V}[u_\alpha] \notag \\
& = \liminf_{\alpha\to 0+} \left( \|\nabla f_\alpha\|_p^p - \int_{\R^d} V\,dx \, f_\alpha(0)^p \right) \notag \\
& \geq E \left( \int_{\R^d} V\,dx \right) \notag \\
& = - \frac{p-d}{p}\ \left(\frac dp\right)^{\frac{d}{p-d}}\ 
\left( \s_{d,p} \int_{\R^d} V(x)\, dx\right)^{\frac{p}{p-d}} \,.
\end{align}
The last equality comes from Lemma \ref{soblemma}. This is the lower bound claimed in Theorem \ref{thm-main-non-crit}.

It remains to prove \eqref{eq:minimizersgoal}. Arguing as in \eqref{aux} we obtain $\| \nabla u_\alpha \|_p^p \leq \alpha \s_{d,p} \| V_+\|_1 \|\nabla u_\alpha\|_p^d$, and therefore
\begin{equation}
\label{eq:kinaprior}
\|\nabla u_\alpha \|_p \leq C \alpha^{\frac{1}{p-d}} \,.
\end{equation}
According to \eqref{sob} this also implies
\begin{equation}
\label{eq:maxaprior}
\|u_\alpha\|_\infty^p \leq C' \alpha^{\frac{d}{p-d}} \,.
\end{equation}
By Morrey's Sobolev inequality there is a constant $\mathcal M=\mathcal M_{d,p}$ such that for all $v\in W^{1,p}(\R^d)$ and all $x, y \in \R^d$ one has
\begin{equation}
\label{eq:morrey}
|v(x)-v(y)| \leq \mathcal M |x-y|^{(p-d)/p} \|\nabla v\|_p \,.
\end{equation}
We now fix $R>0$ and use Morrey's inequality \eqref{eq:morrey} together with \eqref{eq:kinaprior} to get for all $x\in B_R$
$$
|u_\alpha(x) - u_\alpha(0)| \leq \mathcal M R^{\frac{p-d}p} \|\nabla u_\alpha \|_p \leq C_R\ \alpha^{\frac{1}{p-d}}
$$
This, together with \eqref{eq:maxaprior}, yields for all $x\in B_R$
$$
|u_\alpha(x)^p - u_\alpha(0)^p| \leq p\, |u_\alpha(x) - u_\alpha(0)|\, \max\{ u_\alpha(x)^{p-1},u_\alpha(0)^{p-1}\} \leq C_R'\ \alpha^{\frac{p+d(p-1)}{p(p-d)}}
$$
Thus,
\begin{align*}
& \alpha^{-\frac d{p-d}} \left| \int_{\R^d} V(x)\, \left ( u_\alpha(x)^p- u_\alpha(0)^p\right )\, dx \right| \\
& \qquad \leq \alpha^{-\frac d{p-d}}\, \| V\|_1\, \sup_{B_R} | u_\alpha^p- u_\alpha(0)^p | + \alpha^{-\frac d{p-d}} \,2\,\|u_\alpha\|_\infty^p \,\int_{B_R^c} |V| \,dx \\
& \qquad \leq \alpha^{\frac 1p} C_R' \, \| V\|_1 + 2 C' \int_{B_R^c} |V| \,dx \,.
\end{align*}
Letting first $\alpha\to 0$ and then $R\to\infty$ we obtain \eqref{eq:minimizersgoal}. This completes the proof.


\subsection{Convergence of minimizers}

The following theorem about the behavior of the $u_\alpha$ is an (almost) immediate consequence of Lemma \ref{soblemma} and Theorem \ref{thm-main-non-crit} and its proof.

\begin{prop}  \label{prop:minimizers}
Let $p>d$ and let $V\in L^1(\R^d)$ with $\int_{\R^d} V(x) dx>0$.
For $\alpha>0$ let $u_\alpha$ be a
non-negative minimizer of $Q_{\alpha V}[ \cdot]$ with
$\|u_\alpha\|_p=1$ and define $f_\alpha$ by \eqref{scaling}.
Then for any sequence $(\alpha_n)\subset (0,\infty)$ converging to zero there is a subsequence $(\alpha_{n_k})$ and an $f_0\in W^{1,p}(\R^d)$ such that 
$f_{\alpha_{n_k}}\to f_0$ in $W^{1,p}(\R^d)$. Moreover, $f_0$ is
a minimizer of \eqref{eq:e} with $v=\int_{\R^d} V\,dx$.
\end{prop}

We recall that, by the Sobolev embedding theorem and the Rellich--Kondrachov theorem, convergence in $W^{1,p}(\R^d)$ for $p>d$ implies convergence in $L^\infty(\R^d)$ and in $C^{0,(p-d)/p}(\R^d)$.

We also note that if the minimizer of the Sobolev inequality \eqref{sob} is unique (up to translations, dilations and multiplication by constants), then Proposition \ref{prop:minimizers} implies that $f_\alpha$ converges as $\alpha\to 0$ (without passing to a subsequence).

\begin{proof}
It follows from \eqref{eq:lowerboundmin} together with the upper bound in Theorem \ref{thm-main-non-crit} that $(f_\alpha)$ is a minimizing sequence for problem \eqref{eq:e} with $v= \int_{\R^d} V\,dx$. Therefore, the assertion follows from the relative compactness asserted in Lemma \ref{soblemma}.
\end{proof}

\section{Case $d = p$}    \label{sect:p=d} 

Throughout this section we suppose that $p=d$. Similarly as in the case $d<p$ we start with a couple of preliminary lemmas which which will be used to ensure existence of a minimizer of problem \eqref{var-problem}.

\subsection{Preliminary results}

\begin{lem} \label{bounded:critical}
Assume that $V\in L^q(\R^d)$ with some $q >1$. Then $Q_V [u]/ \|u\|_d^d$ is bounded from below and $Q_V[\cdot]$ is weakly lower semi-continuous in $W^{1,p}(\R^d)$.
\end{lem}

Recall that by Sobolev inequalities, see, e.g., \cite{a}, for every $r \in [d, \infty)$ there is a constant $\tilde {\mathcal S}_{d,r}$ such that 
\begin{equation}\label{sobolev-d=p}
\|u\|_r \leq \tilde {\mathcal S}_{d,r} \, \|\nabla u\|_d^\theta \, \|u\|_d^{1-\theta}\,,
\qquad \text{for all}\ u\in W^{1,d}(\R^d) \,.
\end{equation}
Here $0\leq\theta<1$ is defined by $\frac dr = 1-\theta$.

\begin{proof}
H\"older's inequality and \eqref{sobolev-d=p} with $r=dq/(q-1)$ imply that
$$
 \int_{\R^d} V |u|^d \,dx \leq \|V_+\|_q \|u\|_r^d
\leq \|V_+\|_q \, \tilde {\mathcal S}_{d,r} \, \|\nabla u\|_d^{d\theta} \, \|u\|_d^{d(1-\theta)} \,.
$$
Thus,
\begin{align*}
Q_V[u] & \geq \|\nabla u\|_d^d  - \|V_+\|_q \, \tilde {\mathcal S}_{d,r} \, \|\nabla u\|_d^{d\theta} \, \|u\|_d^{d(1-\theta)} \\
& \geq \inf_{X\geq 0} \left( X - \|V_+\|_q \, \tilde {\mathcal S}_{d,r} \, X^{\theta} \, \|u\|_d^{d(1-\theta)} \right) \\
& \geq -C\ \|V_+\|_q^{\frac{1}{1-\theta}} \|u\|_d^d
\end{align*}
where $C>0$ depends only on $d$ and $q$ (through $r$). This proves lower boundedness.

Now let us prove weak lower semi-continuity of $Q_V[u]$. As in the proof of Lemma \ref{weak-cont} it suffices to show that $\int_{\R^d} V|u|^p\,dx$ is weakly continuous on $W^{1,d}(\R^d)$. Assume that $(u_j)$ converges weakly in $W^{1,d}(\R^d)$ to some $u$. Given $\delta>0$ define
$\Omega_\delta=\{x\in\R^d\, :\, |V(x)| > \delta\}$. Since $(u_j)$ is bounded in $L^d(\R^d)$, we have
\begin{equation} \label{delta}
\Big | \int_{\Omega^c_\delta} V (|u|^d-|u_j|^d)\, dx\Big
| \leq \, C\, \delta
\end{equation}
with $C$ independent of $j$. Moreover, the Sobolev inequality
\eqref{sobolev-d=p} implies that $u_j$ is uniformly bounded in
$L^r(\R^d)$ for every $r\in[d, \infty)$. Hence by H\"older inequality 
\begin{align*}
\Big | \int_{\Omega_\delta} V (|u|^d-|u_j|^d)\, dx\Big | & \leq
\|V\|_q \, \left( \int_{\Omega_\delta}
||u|^d-|u_j|^d|^{\frac q{q-1}} \, dx \right)^{\frac{q-1}q}\\
& = \|V\|_q \, \left( \int_{\Omega_\delta} |(|u|-|u_j|)\,
\varphi_j|^{\frac q{q-1}} \, dx
\right)^{\frac{q-1}{q}},
\end{align*}
where for every $r\in [d,\infty)$ there is a $C_r$ such that $\|\varphi_j\|_r \leq C_r$ for all $j$. Since $\Omega_\delta$ has finite measure,
$u_j\to u$ in $L^r(\Omega_\delta)$ for any $r<\infty$ by the Rellich--Kondrashov theorem. (For instance, in \cite[Thm. 8.9]{ll}, the Rellich--Kondrashov theorem is only stated for bounded sets. However, for any $\epsilon>0$ we can find a bounded set $\omega\subset\Omega_\delta$ such that $|\Omega_\delta\setminus\omega|<\epsilon$. Then $u_j\to u$ in $L^r(\omega)$ by the bounded Rellich-Kondrashov theorem and, since $(u_j)$ is bounded in $L^s(\Omega_\delta)$ for some $s>r$, by H\"older $\|u_j\|_{L^r(\Omega_\delta\setminus\omega)} \leq \|u_j\|_{L^s(\Omega_\delta)} \epsilon^{(s-r)/s}$. Thus, $u_j\to u$ in $L^r(\Omega_\delta)$, as claimed.)

We thus conclude, again with $r= 2q/(q-1)$, that
$$
\int_{\Omega_\delta} |(|u|-|u_j|)\, \varphi_j|^{\frac q{q-1}}
\, dx \, \leq \, C_r^{\frac q{q-1}} \, \left(\int_{\Omega_\delta}
|u-u_j|^{\frac{2q}{q-1}} \, dx \right)^{1/2}\, \to \, 0
\quad \text{as\, } j\to\infty.
$$
This in combination with \eqref{delta} proves the claimed weak continuity.
\end{proof}


\subsection{Proof of Theorem \ref{thm-main-crit}. Upper bound}


\begin{prop} \label{lowerb-prop}
Let $V\in L^1(\R^d)$ be such that $\int_{\R^d} V(x)\, dx  > 0$. Then 
\begin{equation} \label{lowerb-limit}
\limsup_{\alpha\to 0+}\, \alpha^{\frac{1}{d-1}}\, \log \frac1{|\lambda(\alpha V)|} \, \leq \, d\, \omega_d^{\frac{1}{d-1}}\, 
\Big(\int_{\R^d} V(x)\, dx \Big)^{-\frac{1}{d-1}}.
\end{equation}
\end{prop}

\begin{proof}
Let $\beta > 1$ and consider the family of test functions $v_\beta$ defined by
\begin{equation}
v_\beta(x)  = 1 \quad \text{if\, \, } |x| \leq 1, \qquad 
v_\beta(x)  =
\left(1- \frac{\log |x|}{\log \beta}\right)_+ \quad \text{if\,
} |x| >1 \,.
\end{equation}
Then $v_\beta \in W^{1,d}(\R^d)$ and, since $0\leq v_\beta\leq \chi_{\{|\cdot|<\beta\}}$, we have
$$
\|v_\beta\|_d^d \ \leq \,c\, \beta^d
$$
for all $\beta>1$ with a constant $c>0$ depending only on $d$. Moreover,
$$
Q_{\alpha V}[v_\beta] \leq \omega_d\,
(\log \beta)^{1-d} -\alpha\, \int_{\R^d} V(x)\, dx + \alpha R_\beta
$$
with
$$
R_\beta = \int_{\{|x|>1\}} V_+ \left( 1- \left( 1 - \frac{\log |x|}{\log \beta} \right)_+ \right) dx \,.
$$
By dominated convergence, $R_\beta \to 0$ as $\beta\to\infty$.

Let $\eps>0$ be given and choose $\beta_\eps>1$ such that
$$
R_\beta \leq \eps \int_{\R^d} V \,dx
\qquad \text{for all}\ \beta\geq \beta_\eps \,.
$$
Now, for any $\alpha$, define
$$
\beta(\alpha) = \exp\left( \left(\frac{\omega_d}{\alpha(1-\epsilon)\int_{\R^d} V\,dx} \right)^{1/(d-1)} \right) \,.
$$
Note that $\beta(\alpha)>1$ and that
$$
\frac{\omega_d}{(\log\beta(\alpha))^{d-1}} - \alpha(1-\eps) \int_{\R^d} V\,dx = 0 \,.
$$
Define $\alpha_\eps>0$ by $\beta(\alpha_\eps)= \beta_\eps$. Then for $\alpha\leq\alpha_\eps$ our upper bound on $Q_{\alpha V}[v_\beta]$ is non-positive and therefore
\begin{align} \label{first-ub}
\lambda(\alpha V) & \leq \frac{Q_{\alpha V}[v_{\beta(\alpha)}]}{\|u_{\beta(\alpha)}\|_d^d} \notag \\
& \leq \, c^{-1} \, \beta(\alpha)^{-d}\, \left(\omega_d\,
(\log \beta(\alpha))^{1-d} -\alpha\, \int_{\R^d} V(x)\,dx +\alpha R_\beta \right) \notag \\
& = - c^{-1} \alpha \left( \epsilon \int_{\R^d} V\,dx -R_{\beta(\alpha)} \right) \exp\left( -d \left(\frac{\omega_d}{\alpha(1-\epsilon)\int_{\R^d} V\,dx} \right)^{1/(d-1)} \right)      .
\end{align}
This implies 
$$
\limsup_{\alpha\to 0+}\, \alpha^{\frac{1}{d-1}}\, \log \frac1{|\lambda(\alpha V)|} \, \leq \, d\, \omega_d^{\frac{1}{d-1}}\, \Big( (1-\eps)\int_{\R^d} V(x)\, dx\Big)^{-\frac{1}{d-1}}.
$$
By letting $\eps\to 0$ we arrive at \eqref{lowerb-limit}.  
\end{proof}

\begin{cor} \label{cor-local}
Let $V$ satisfy assumptions of Lemma \ref{bounded:critical}. Then
for every $\alpha>0$  there exists a locally bounded positive function
$u_\alpha\in W^{1,d}(\R^d)$ such that 
$\lambda(\alpha V)\, \|u_\alpha\|_d^d = Q_{\alpha V}[u_\alpha]$.
\end{cor}

\begin{proof}
Inequality \eqref{first-ub} with $\beta$ large enough shows that  $\lambda(\alpha V)<0$ for all  $\alpha>0$. 
Hence the existence of a non-negative minimizer $u_\alpha$ follows from Lemma \ref{bounded:critical} in the same way as
in the case $d <p$. Since $u_\alpha$ is a non-negative weak solution of \eqref{eq:groundstate}, the Harnack inequality \cite[Thm. 6]{s} implies that $u_\alpha$ is locally bounded and positive. 
\end{proof}


\subsection{Proof of Theorem \ref{thm-main-crit}. Lower bound}


\subsubsection*{The case of positive $V$}

\begin{prop}\label{lowerb-lemmapos}
Assume that $0\leq V\in L^q(\R^d)\cap L^1(\R^d)$ for some $q>1$ with $V\not\equiv 0$. Then there are $\alpha_0>0$ and $C>0$ such that for all $0<\alpha\leq\alpha_0$ we have 
\begin{equation} \label{lower-bpos}
\lambda(\alpha V) \geq - C \, \alpha^{-1}\, \exp\left[- \left(\frac{d^{d-1}\, \omega_d}{ \alpha\, \int_{\R^d} V \,dx } \right)^{\frac{1}{d-1}}\right] \,.
\end{equation}
\end{prop}

\begin{proof}
Let $V^*$ be the symmetric decreasing rearrangement of $V$. 
Since $\int_{\R^d} V\,dx = \int_{\R^d} V^*\,dx$, $\int_{\R^d} V^q \,dx = \int_{\R^d} (V^*)^q\,dx$ and, by rearrangement inequalities (see, e.g.,~\cite{tal} and \cite[Thm. 3.4]{ll}),
$$
\lambda(\alpha V) \geq \lambda(\alpha V^*) \,,
$$
we may and will assume in the following that $V=V^*$.

By Corollary \ref{cor-local} there is a minimizer $u_\alpha$ of $Q_{\alpha V}[u]/\|u\|_d^d$. Again, by rearrangement inequalities, we may assume that $u_\alpha$ is a radially symmetric function which is non-increasing with respect to the radius. Let $\rho>0$ be an arbitrary parameter. (In this proof there is no loss in assuming that $\rho=1$, but in the proof of Proposition \ref{lowerb-lemmagen} we will repeat the argument with a general $\rho$.) We normalize $u_\alpha$ such that
$$
u_\alpha(x) = u_\alpha(|x|) =1, \qquad \text{for all}\ x\in\R^d\ \text{with}\ |x|=\rho \,.
$$
Let $R\geq 2\rho$ be a parameter to be specified later and let $\chi$ be defined by
$$
\chi(r) = 1 \quad \text{if \ } 0\leq r \leq \rho, \qquad \chi(r) = \Big(1-\frac{r-\rho}{R-\rho}\Big)_+ \quad \text{if \ } r>\rho \,.
$$
Then for any $\eps \in(0,1]$ we have
\begin{align*}
\|\nabla( \chi\, u_\alpha)\|_d^d  & \leq (1+\eps) \|\chi\, \nabla u_\alpha\|_d^d +c\, \eps^{1-d}\, \| u_\alpha\, \nabla \chi\|_d^d \\
& \leq  (1+\eps) \| \nabla u_\alpha\|_d^d +c'\, \eps^{1-d}\, R^{-d}\,  \| u_\alpha\|_d^d \,,
\end{align*}
and therefore
\begin{equation}
\|\nabla u_\alpha\|_d^d \geq \|\nabla (\chi\, u_\alpha)\|_d^d/(1+\eps)  - c''\,  \eps^{1-d}\, R^{-d}\, \|u_\alpha\|_d^d \,.
\end{equation}
Since $\chi\, u_\alpha$ has support in the ball of radius of radius $R$ and is bounded from below by one on the ball of radius $\rho$, the formula for the capacity of two nested balls \cite[Sec. 2.2.4]{m} gives 
\begin{equation}\label{kinetic-lowerbpos}
\|\nabla u_\alpha\|_d^d \geq \, \frac{\omega_d\, ( \log (R/\rho))^{1-d}}{1+\eps} - c''\,  \eps^{1-d}\, R^{-d}\, \|u_\alpha\|_d^d \,.
\end{equation}
Moreover, since $|u_\alpha(x)| \leq 1$ for $|x| >1$, we obtain
\begin{equation} \label{first-lowbpos}
\lambda(\alpha V) \geq  \frac{\omega_d\, 
(\log (R/\rho) )^{1-d} - (1+\eps)\, \alpha \left( \int_{B_1} V u_\alpha^d \,dx + \int_{B_1^c} V \,dx \right)}{(1+\eps)\, \|u_\alpha\|_d^d} 
- \frac{c''}{\eps^{d-1}\, R^{d}} \,.
\end{equation}

We next claim that there are constants $C>$ and $\alpha_0>0$ such that for all $0<\alpha\leq \alpha_0$,
\begin{equation}
\label{eq:oscilcritpos}
\sup_{B_\rho} \left( u_\alpha^d -1 \right) \leq C \alpha^{\frac{1}{d-1}} \,.
\end{equation}
Accepting this for the moment and returning to \eqref{first-lowbpos} we obtain
$$
\lambda(\alpha V) \geq  \frac{\omega_d\, 
(\log  (R/\rho) )^{1-d} - (1+\eps)\left(1+ C \alpha^{\frac{1}{d-1}}\right) \, \alpha \int_{\R^d} V \,dx}{(1+\eps)\, \|u_\alpha\|_d^d} 
- \frac{c''}{\eps^{d-1}\, R^{d}} \,.
$$
For given $0<\epsilon\leq 1$ and $0<\alpha\leq\alpha_0$ we now choose
$$
R = \rho \exp \left( \left( \frac{\omega_d}{(1+\eps)\left(1+ C \alpha^{\frac{1}{d-1}}\right) \, \alpha \int_{\R^d} V \,dx} \right)^{\frac{1}{d-1}} \right)
$$
so that
$$
\lambda(\alpha V) \geq
- \frac{c''}{\eps^{d-1} \rho^d}\, \exp \left(-d \left( \frac{\omega_d}{(1+\eps)\left(1+ C \alpha^{\frac{1}{d-1}}\right) \, \alpha \int_{\R^d} V \,dx} \right)^{\frac{1}{d-1}} \right)   \,.
$$
Finally, we choose $\epsilon=C \alpha^{\frac{1}{d-1}}$ to obtain
\begin{equation}
\lambda(\alpha V) \geq 
- \frac{c'''}{\alpha}\, \exp \left(-d \left( \frac{\omega_d}{\left(1+ C' \alpha^{\frac{1}{d-1}}\right) \, \alpha \int_{\R^d} V \,dx} \right)^{\frac{1}{d-1}} \right) \,.
\end{equation}
Up to increasing $c'''$ this implies the statement of the proposition.

Thus, it remains to prove \eqref{eq:oscilcritpos}. For simplicity we give the proof only for $\rho=1$ (which is enough for the proof of the proposition). We apply Alvino's version of the Moser--Trudinger inequality \cite{alv} to the function $u_\alpha-1$ and obtain
\begin{equation} \label{mos-trpos}
0 \, < \, u_\alpha(r) -1 \, \leq \, C\, \| \nabla u_\alpha\|_{L^d(B_1)} \, |\log r|^{\frac {d-1} d}\,, \qquad r\leq 1.
\end{equation}
Using this upper bound on $u_\alpha$ we arrive at
\begin{align*}
\|\nabla u_\alpha \|^d_{L^d(B_1)} & \leq \|\nabla u_\alpha\|_d^d \\
& \leq \alpha \int_{\R^d} V |u_\alpha|^d \,dx \\
& \leq \alpha 2^{d-1} \left(  \| V \|_{L^1(B_1)} + C \|\nabla u_\alpha \|^d_{L^d(B_1)} \omega_d \int_0^1 V(r) | \log r |^{d-1} r^{d-1} \, dr \right) .
\end{align*}
The assumption $V\in L^q(\R^d)$ for some $q>1$ implies that $V \in L^1(B_1, |\log |x| |^{d-1} \, dx)$, and therefore there is a $C'>0$ and an $\alpha_0>0$ such that for all $0<\alpha \leq\alpha_0$
$$
\|\nabla u_\alpha \|_{L^d(B_1)}^d \leq C' \, \alpha^{1/d} \,.
$$
Re-inserting this into \eqref{mos-trpos}, we find for all $0<\alpha \leq\alpha_0$
\begin{equation} \label{mos-tr-2}
0 \, < \, u_\alpha(r) -1 \, \leq \, C'' \,  \alpha^{1/d} \, |\log r|^{\frac {d-1} d}\,, \qquad r\leq 1.
\end{equation}
Hence the minimizer $u_\alpha$  satisfies for all $0<r\leq 1$,
\begin{align}
((-r\, u_\alpha' (r))^{d-1})' & = \alpha \, V(r)\, u_\alpha(r)^{d-1} \, r^{d-1} + \lambda(\alpha)\,  u_\alpha(r)^{d-1} r^{d-1}  \label{eq:radial} \\
& \leq \alpha \, V(r)\,  r^{d-1}\, \big( 1+C''\alpha^{\frac1{d}}\,  |\log r|^{\frac {d-1} d}\big)^{d-1} \nonumber
\end{align}
and
\begin{align}
((-r\, u_\alpha' (r))^{d-1})' & = \alpha \, V(r)\, u_\alpha(r)^{d-1} \, r^{d-1} + \lambda(\alpha)\,  u_\alpha(r)^{d-1} r^{d-1}  \label{eq:radial2} \\
& \geq \lambda(\alpha)\, r^{d-1} \big( 1+C''\alpha^{\frac1{d}}\,  |\log r|^{\frac {d-1} d}\big)^{d-1} \nonumber \,.
\end{align}
Since the right hand sides of \eqref{eq:radial} and \eqref{eq:radial2} are integrable with respect to $r$ (for \eqref{eq:radial} we use here again the assumption that $V\in L^1(\R^d)\cap L^q(\R^d)$ for some $q>1$), the function $(-r\, u_\alpha' (r))^{d-1}$ has a finite limit as $r\to 0$.  Since $u_\alpha \in W^{1,d}(\R^d)$, it follows that  this limit must be zero. Thus, from \eqref{eq:radial} we get for all $0<r\leq 1$
\begin{align*}
(-r\, u_\alpha' (r))^{d-1} & \leq \alpha \int_0^r V(s) \, s^{d-1}\,  
\big( 1+C''\alpha^{\frac1{d}}\,  |\log s|^{\frac {d-1} d}\big)^{d-1} \,ds \\
& \leq \alpha \|V\|_{L^q(B_1)} \left( \int_0^r \, s^{d-1}\,  
\big( 1+C''\alpha^{\frac1{d}}\,  |\log s|^{\frac {d-1} d}\big)^{q'(d-1)} \,ds \right)^{1/q'} \\
& \leq C''' \alpha \|V\|_{L^q(B_1)} r^{d/q'} \left(1+ |\log r| \right)^{\frac{(d-1)^2}{d}} \,.
\end{align*}
Finally, this implies that
\begin{align*}
u_\alpha(r) - 1 & = -\int_r^1 u_\alpha'(s) \, ds \\
& \leq \left( C''' \alpha \|V\|_{L^q(B_1)} \right)^{\frac 1{d-1}} 
\int_r^1 s^{\frac d{q'(d-1)}} \left(1+ |\log s| \right)^{\frac{(d-1)}{d}} \frac{ds}{s} \,.
\end{align*}
Since the integral on the right side converges, we have shown \eqref{eq:oscilcritpos}. This completes the proof of the lemma.
\end{proof}


\subsubsection*{The case of compactly supported $V$}

\begin{prop} \label{lowerb-lemmagen}
Let $V$ be a function with compact support, $\int_{\R^d} V(x)\, dx  >0$ and $V\in L^q(\R^d)$ for some $q>1$. Then there are $\alpha_0>0$ and $C>0$ such that for all $0<\alpha\leq\alpha_0$ we have 
\begin{equation} \label{lower-bgen}
\lambda(\alpha V) \geq - \exp\left[- \left(\frac{d^{d-1}\, \omega_d}{ \alpha\, \int_{\R^d} V \,dx\ (1+C\alpha^{\frac1{d}}) } \right)^{\frac{1}{d-1}}\right] \,.
\end{equation}
\end{prop}

Similarly as in the case $d < p$ a key ingredient in the proof is to show that minimizers, when suitably normalised, converge locally to a constant function. In the case $d<p$ we deduced this from Morrey's inequality. Here the argument is considerably more complicated and based on Harnack's inequality for quasi-linear equations. We shall prove

\begin{lem} \label{lem-hoelder}
For each $d\in\N$, $q>1$ and $M>0$ there are constants $C>0$ and $\beta\in (0,1)$ with the following property. Let $\rho>0$ and assume that $W\in L^q_{\mathrm loc}(\R^d)$ with $W \leq 0$ in $B_{5\rho}^c$ and $\rho^{d-\frac dq} \|W\|_{L^q(B_{15\rho})}\leq M$. Then, if $u\in W^{1,d}(\R^d)$ is a positive, weak solution of the equation $-\Delta_d(u)=W u^{d-1}$ in $\R^d$ satisfying $\inf_{B_{5\rho}} u\leq 1$ and if $y\in\R^d$ and $r>0$ are so that $B(3r,y)\subset B_{3\rho}$, we have
\begin{equation} \label{hoelder-cont}
\sup_{B(r,y)} u - \inf_{B(r,y)} u \leq C\, \|W\|_{L^q(B_{5\rho})}^{1/d} \ \rho^{1-\frac{1}{q}-\beta}\ r^\beta \,.
\end{equation} 
\end{lem}

The point of this lemma is that the dependence of $W$ enters explicitly on the right side of \eqref{hoelder-cont}. In our application, we will have $\|W\|_{L^q(B_{5\rho})}\to 0$, and therefore Lemma \ref{lem-hoelder} shows that the oscillations of $u$ vanish with an explicit rate.

We recall that $u$ is a weak solution of $-\Delta_d(u)=W |u|^{d-2}u$ in $\R^d$ if
\begin{equation} \label{weak-eq}
\int_{\R^d} |\nabla u|^{d-2}\, \nabla u \cdot \nabla\varphi \, dx = \int_{\R^d} W |u|^{d-2}\, u\, \varphi \, dx 
\end{equation}
for any $\varphi\in W^{1,d}(\R^d)$. 

The following lemma, whose proof can be found, for instance, in \cite{mo1,mo2} or \cite[Lem. 2.4.1]{lu}, plays a key role in the proof of Lemma \ref{lem-hoelder}.

\begin{lem} \label{lem-morrey}
Let $\Omega \subseteq \R^d$ be open and assume that $u\in W^{1,d}(\Omega)$ is such that there are constants $K>0$ and $\beta>0$ such that for all $y\in\Omega$ and $r>0$ with $B(r,y)\subset\Omega$ one has
\begin{equation}  \label{cond-morrey}
\int_{B(r,y) } |\nabla u|^d\, dx \, \leq \, K \, r^{\beta d} \,.
\end{equation}
Then for all $y\in\Omega$ and $r>0$ such that  $B(3r/2, y)\subset\Omega$ we have  
\begin{equation}
\sup_{B(r/2, y) } u - \inf_{B(r/2, y) } u \, \leq \,  \frac{4}{\beta}\, \left(\frac{K}{\omega_d}\right)^{\frac 1d}\, r^\beta .
\end{equation} 
\end{lem}

\begin{proof}[Proof of Lemma \ref{lem-hoelder}]
By the Harnack inequality \cite[Thm.6]{s} there is a constant $C_1$, which depends only on $d$, $q$ and an upper bound on
$\rho^{d-\frac dq}\, \|W\|_{L^q(B_{15\rho})}$ such that 
\begin{equation*}
\sup_{B_{5\rho}} u \leq\, C_1\, \inf_{B_{5\rho}} u \,.
\end{equation*}
Since $\inf_{B_{5\rho}} u(x) \leq 1$, we conclude that
\begin{equation} \label{maxu}
\sup_{B_{5\rho}} u(x) \leq\, C_1 \,.
\end{equation}

Our goal is to apply Lemma \ref{lem-morrey} with $\Omega=B_{3\rho}$. We have to verify condition \eqref{cond-morrey} for some $K$ and $\beta$. First, note that
\begin{equation} \label{unit}
\int_{\R^d} |\nabla u|^d \,dx \,
=\, \int_{\R^d} W u^d \,dx\,
\leq\, \int_{B_{5\rho}} W \, u^d\, dx \,
\leq \, \omega_d^{1-\frac{1}{q}} (5\rho)^{d-\frac{d}{q}} \|W\|_{L^q(B_{5\rho})} \, C_1^d = c_1 \mathcal N\,,
\end{equation}
where we have set $c_1 = \omega_d^{1-\frac{1}{q}} 5^{d-\frac{d}{q}}$ and
\begin{equation}
\label{eq:n}
\mathcal N = \rho^{d-\frac{d}{q}} \|W\|_{L^q(B_{5\rho})} \, C_1^d \,.
\end{equation}
Hence, for any $\beta>0$, \eqref{cond-morrey} holds for any ball $B(r,y)\subset B_{3\rho}$ with $r\geq\rho$ provided we choose the constan $K$ at least as big as $c_1 \mathcal N \rho^{-\beta d}$.

Thus, it remains to verify \eqref{cond-morrey} for $r<\rho$. Let $0\leq \zeta\leq 1$ be a radial function with support in $\overline{B_2}$ which is $\equiv 1$ on $B_1$ and satisfies $|\nabla\zeta|\leq 1$. Let $y$ and $s$ be such that $B(2s,y)\subset B_{5\rho}$. We choose the test function $\varphi(x) = \zeta(|x-y|/s) (u(x)-a)$ in \eqref{weak-eq}, where the parameter $a$ will be specified later. This gives the inequality  
\begin{align} \label{test-a}
\int_{B(s,y)} |\nabla u|^d\, dx
& \leq \int_{\R^d} \zeta(|x-y|/s) |\nabla u|^d\, dx \notag \\
& \leq \int_{B(2s,y)} |W| \, u^{d-1}\, |u-a| \, dx  + s^{-1}\, \int_{A(s,y)}  |\nabla u|^{d-1} |u-a|\, dx \,.
\end{align}
with $A(s,y) = B(2s,y) \setminus B(s,y)$. Now we set $a = \frac{1}{|A(s,y)|} \int_{A(s,y)} u\, dx$, where $|A(s,y)|$ denotes the Lebesgue measure of $A(s,y)$. By the H\"older and Poincar\'e inequalities,
\begin{align*}
\int_{A(s,y)}  |\nabla u|^{d-1} |u-a|\, dx 
& \leq \Big(\int_{A(s,y)}  |\nabla u|^{d} \,  dx\Big)^{\frac{d-1}{d}}\, \Big(\int_{A(s,y)} |u-a|^d\, dx \Big)^{\frac 1d} \\
& \leq C^{\mathrm P}\,  s  \int_{A(s,y)}  |\nabla u|^d \,  dx \,,
\end{align*}
where $C^{\mathrm P}$ is the constant in the Poincar\'e inequality in $A(1,0)$. By scaling one easily sees that the Poincar\'e constant in $A(s,y)$ is given by $C^{\mathrm P} s$. This fact was used in the previous bound.

Let us bound the first term on the right side of \eqref{test-a}. Since both $u$ and $|a|$ are bounded from above by $C_1$ on $B(2s,y)$, see \eqref{maxu}, we have
$$
\int_{B(2s,y)} |W| \, u^{d-1}\, |u-a| \, dx \leq \| W \|_{L^1(B(2s,y))} 2 C_1^p
\leq c_2\, \mathcal N\, (s/\rho)^{d-\frac dq} \,,
$$
where $c_2 = \omega_d^{1-\frac 1q}\, 2^{d+1-\frac dq}$.

Thus, \eqref{test-a} implies
$$
\int_{B(s,y)} |\nabla u|^d\, dx  \leq \, c_2\, \mathcal N\, (s/\rho)^{d-\frac dq} + C^{\mathrm P}\, \int_{A(s,y)}  |\nabla u|^d \,  dx,
$$
where $c_1 = 2^{d+1-\frac dq}\, \omega_d^{1-\frac 1q}$. Adding $C^{\mathrm P}\, \int_{B(s,y)} |\nabla u|^d\, dx$ to both sides of the above inequality we arrive at 
\begin{equation} \label{iterate}
\int_{B(s,y)} |\nabla u|^d\, dx  \leq \, c_3\, \mathcal N\, (s/\rho)^{d-\frac dq} + \kappa\, \int_{B(2s,y)} |\nabla u|^d\, dx,
\end{equation}
with $c_3= c_2/(1+C^{\mathrm P})$ and 
$$
\kappa = \frac{C^{\mathrm P}}{1+C^{\mathrm P}} <1.
$$ 
To simplify the notation, we introduce the shorthand $D(s)=  \int_{B(s,y)} |\nabla u|^d\, dx$.  Iterating inequality \eqref{iterate} gives 
$$
D( 2^{-n} s)\, \leq \,  c_3\,\mathcal N\,(s/\rho)^{d-\frac dq}\, 2^{n(\frac dq-d)}\, \sum_{j=0}^{n-1} \big(\kappa\,  2^{d-\frac dq}\big)^j +   \kappa^n \, D(s)
$$
for all $n\in\N$ and every $s>0$ such that $B(s,y)\subset B_{5\rho}$. Next, we sum the geometric series on the right side and obtain a $c_4$ and a $\mu<1$ (both depending only on $d$ and $q$) such that
$$
2^{n(\frac dq-d)}\, \sum_{j=0}^{n-1} \big(\kappa\,  2^{d-\frac dq}\big)^j
\leq c_4\, \mu^n 
\qquad\text{for all}\ n\in\N \,.
$$
Thus, recalling \eqref{unit},
\begin{equation} \label{iterated}
D( 2^{-n} s) \, \leq \, \left( c_3 c_4 \, (s/\rho)^{d-\frac dq} + c_1\right)\mathcal N \, \max\{\mu^n,\kappa^n\} 
\end{equation}
for all $n\in\N$ and all $s$ such that $B(s,y)\subset B_{5\rho}$.

Now let $B(r,y)\subset B_{3\rho}$ with $r<\rho$. There are $k\in\N$ and $t\in [1,2)$  such that $2^{-k-1}\, t\rho < r \leq 2^{-k}\, t\rho$. Since $B(t\rho,y)\subset B_{5\rho}$ we may apply inequality \eqref{iterated} with $k=n$ and $s=t\rho$ to get
\begin{align*}
\int_{B(r,y)} |\nabla u|^d\, dx 
& \leq D( 2^{-k} t\rho) \\
& \leq\, 
\left( c_3 c_4 \, t^{d-\frac dq} + c_1 \right) \mathcal N\,
\max\{\mu^k,\kappa^k\} \\
& \leq \left( c_3 c_4 \, 2^{d-\frac dq} + c_1 \right) \mathcal N
\left( \frac{2r}{\rho} \right)^{\beta d}
\quad\text{with}\ \beta = - \frac{\log \max\{\mu,\kappa\}}{d\, \log 2} >0 \,.
\end{align*}

To summarize, we have shown that \eqref{cond-morrey} holds for any $B(r,y)\subset B_{3\rho}$ with the above choice of $\beta$ and with
$$
K = \max\left\{ c_1, \left( c_3 c_4 \, 2^{d-\frac dq} + c_1 \right) 2^{\beta d} \right\} \mathcal N \rho^{-\beta d} \,.
$$
Here $c_1$, $c_3$ and $c_4$ depend only on $d$ and $q$, and $\mathcal N$ was defined in \eqref{eq:n}. In view of Lemma \ref{lem-morrey} this proves \eqref{hoelder-cont}.
\end{proof}

\begin{proof}[Proof of Proposition \ref{lowerb-lemmagen}]
The beginning of the proof is identical to that of Proposition \ref{lowerb-lemmapos}. Let $\rho>0$ be such that the support of $V$ is contained in $\overline{B_{5\rho}}$. We let again $u_\alpha$ be a minimizer of $Q_{\alpha V}[u]/\|u\|_d^d$. From Corollary \ref{cor-local} we know that $u_\alpha$ can be chosen strictly positive and therefore we may normalize $u_\alpha$ by $\inf_{B_\rho} u_\alpha =1$. Arguing exactly as before we arrive at the following variant of \eqref{first-lowbpos},
\begin{equation} \label{first-lowb}
\lambda(\alpha V) \geq  \frac{\omega_d\, 
(\log  (R/\rho) )^{1-d} - (1+\eps)\, \alpha \int_{\R^d} V\, |u_\alpha|^d\,dx}{(1+\eps)\, 
\|u_\alpha\|_d^d} \, -c'' \, \eps^{1-d}\, R^{-d} \,.
\end{equation}

We now claim that there is a constant $C>0$ (depending on $d$, $q$, $V$, but not on $\alpha$) such that
\begin{equation} \label{mos-tr}
| u_\alpha(x) -1 | \leq C\, \alpha^{\frac 1d} \qquad
\text{for all}\ x\in B_\rho \,.
\end{equation}
Indeed, this follows from Lemma \ref{lem-hoelder} applied to $W=\alpha V+ \lambda(\alpha V)$ and $u=u_\alpha$ with $B(r,y) = B_\rho$. Note that we indeed have $\inf_{B_{5\rho}} u_\alpha \leq \inf_{B_\rho} u_\alpha = 1$. Moreover, we use the fact that $\lambda(\alpha V)\geq -C \alpha$, which follows easily from the bounds in Lemma \ref{bounded:critical}.

With a similar choice as in Lemma \ref{lowerb-lemmapos} for $R$ we obtain
$$
\lambda(\alpha V) \geq - \frac{c''}{\eps^{d-1} \rho^d}
\exp \left( -d \left( \frac{\omega_d}{(1+\eps) \left( 1+ C\,\alpha^{\frac 1d}\right)\alpha \int_{\R^d} V\,dx} \right)^{\frac{1}{d-1}} \right) \,.
$$
Choosing $\epsilon=C \alpha^{\frac 1d}$ we obtain
$$
\lambda(\alpha V) \geq - \frac{c'''}{\alpha^{\frac{d-1}d}}
\exp \left( -d \left( \frac{\omega_d}{\left( 1+ C'\,\alpha^{\frac 1d}\right)\alpha \int_{\R^d} V\,dx} \right)^{\frac{1}{d-1}} \right) \,.
$$
This implies the statement of the proposition.
\end{proof}


\subsubsection*{The general case}

We can finally give the

\begin{proof}[Proof of Theorem \ref{thm-main-crit}]
We use an approximation argument and fix $\varepsilon\in(0,1)$ and $R>0$. Define $V_< = V \chi_{\{|\cdot|<R\}}$ and $V_> = V_+ \chi_{\{|\cdot|\geq R\}}$. Then the inequality
$$
Q_{\alpha V}[u] \geq (1-\eps) Q_{(1-\eps)^{-1} \alpha V_<}[u] + \eps Q_{\eps^{-1} \alpha V_>}[u]
$$
for every $u\in W^{1,d}(\R^d)$ implies
$$
\lambda(\alpha V) \geq (1-\eps) \lambda\left(\frac{\alpha}{1-\eps}\, V_<\right) + \eps \lambda\left(\frac\alpha\eps V_>\right) \,.
$$
Thus,
\begin{align*}
\log \frac{1}{|\lambda(\alpha V)|} & \geq \log \frac{1}{(1-\epsilon)\ |\lambda((1-\epsilon)^{-1} \alpha V_<)|} - \log \left( 1+ \frac{\epsilon\ |\lambda(\epsilon^{-1} \alpha V_>)|}{(1-\epsilon)\ |\lambda((1-\epsilon)^{-1} \alpha V_<)|} \right) \\
& \geq \log \frac{1}{(1-\epsilon)\ |\lambda((1-\epsilon)^{-1} \alpha V_<)|} - \frac{\epsilon\ |\lambda(\epsilon^{-1} \alpha V_>)|}{(1-\epsilon)\ |\lambda((1-\epsilon)^{-1} \alpha V_<)|} \,. 
\end{align*}
From now on we consider $R$ so large that $\int_{B_R} V\,dx>0$. It then follows from Proposition \ref{lowerb-lemmagen} that
\begin{align*}
\liminf_{\alpha\to 0+} \alpha^{\frac{1}{d-1}} \log \frac{1}{(1-\epsilon)|\lambda((1-\epsilon)^{-1} \alpha V_<)|} 
& \geq (1-\epsilon)^{\frac1{d-1}} d\ \omega_d^{\frac{1}{d-1}}\, \left(\int_{B_R}  V(x)\, dx  \right)^{-\frac{1}{d-1}} \,.
\end{align*}
On the other hand, we recall from Proposition \ref{lower-bpos} that there are constants $C>0$ and $\alpha_0>0$ such that for all $0<\alpha\leq\alpha_0\epsilon $,
$$
\lambda(\epsilon^{-1} \alpha V_>)
\geq
- C\epsilon \alpha^{-1}  \exp\left(- \left(\frac{\epsilon d^{d-1}\, \omega_d}{ \alpha\, \int_{B_R^c} V_+ \,dx} \right)^{\frac{1}{d-1}} \right)
$$
Moreover, we recall from Proposition \ref{lowerb-prop} that for every $\delta\in (0,1)$ there are constants $C_\delta>0$ and $\alpha_\delta$ such that for all $0<\alpha\leq\alpha_\delta(1-\epsilon)$,
\begin{align}
\lambda((1-\epsilon)^{-1} \alpha V_<) 
& \leq - (1-\epsilon)^{-1} \alpha\  C_\delta \exp\left( - \left(\frac{(1-\epsilon) d^{d-1} \omega_d}{\alpha(1-\delta)\int_{B_R} V\,dx} \right)^{\frac1{d-1}} \right) \,.
\end{align}
Thus, for $\alpha \leq\min\{\alpha_0\epsilon,\alpha_\delta(1-\epsilon)\}$,
$$
\frac{|\lambda(\epsilon^{-1} \alpha V_>)|}{|\lambda((1-\epsilon)^{-1} \alpha V_<)|}
\leq \frac{C\epsilon(1-\epsilon)}{C_\delta \alpha^2}
\exp\left(- \left(\frac{\epsilon d^{d-1}\, \omega_d}{ \alpha\, \int_{B_R^c} V_+ \,dx} \right)^{\frac{1}{d-1}} + \left(\frac{(1-\epsilon)d^{d-1}\omega_d}{\alpha(1-\delta)\int_{B_R} V\,dx} \right)^{\frac1{d-1}} \right)
$$
For every fixed $\epsilon$ and $\delta$ there is an $R_0>0$ such that for all $R>R_0$,
$$
\frac{\epsilon}{\int_{B_R^c} V_+ \,dx} > \frac{1-\epsilon}{(1-\delta)\int_{B_R} V\,dx} \,.
$$
Thus, for all $R>R_0$ we have
$$
\lim_{\alpha\to 0} \frac{|\lambda(\epsilon^{-1} \alpha V_>)|}{|\lambda((1-\epsilon)^{-1} \alpha V_<)|} = 0 \,.
$$

To summarize, we have shown that for all $\epsilon\in (0,1)$ and for all $R>R_0$,
$$
\liminf_{\alpha\to 0+} \alpha^{\frac{1}{d-1}} \log \frac{1}{|\lambda(\alpha V)|} \geq (1-\epsilon)^{\frac1{d-1}} d\ \omega_d^{\frac{1}{d-1}}\, \left(\int_{B_R}  V(x)\, dx  \right)^{-\frac{1}{d-1}} \,.
$$
Letting $\epsilon\to 0$ and $R\to\infty$ we obtain the theorem.
\end{proof}


\section*{Acknowledgements}

Partial financial support through Swedish research council grant FS-2009-493 (T. E.), U.S. National Science Foundation grant PHY-1347399 (R. F.) and grant MIUR-PRIN08 grant for the project ``Trasporto ottimo di massa, disuguaglianze geometriche e funzionali e applicazioni" (H. K.) is acknowlegded.


\end{document}